\documentclass[a4paper,11pt]{article}
\usepackage{amsmath,amssymb,amsthm,mathtools,mathrsfs}
\usepackage{cite}
\usepackage{microtype}
\usepackage{enumitem}
\usepackage[colorlinks=true,linkcolor=blue,citecolor=blue,urlcolor=blue]{hyperref}
\usepackage[capitalize,noabbrev]{cleveref}
\allowdisplaybreaks

\hypersetup{
  pdftitle={Heat Codimension and Scalar Curvature},
  pdfauthor={Bang-Xian Han},
  pdfsubject={Pointed Nash entropy, synthetic scalar curvature, and macroscopic dimension},
  pdfkeywords={scalar curvature, heat codimension, heat kernel, pointed Nash entropy, macroscopic dimension}
}

\numberwithin{equation}{section}

\newtheorem{theorem}{Theorem}[section]
\newtheorem{proposition}[theorem]{Proposition}
\newtheorem{lemma}[theorem]{Lemma}
\newtheorem{corollary}[theorem]{Corollary}
\theoremstyle{definition}
\newtheorem{definition}[theorem]{Definition}
\theoremstyle{remark}
\newtheorem{remark}[theorem]{Remark}
\crefalias{definition}{definition}
\crefalias{proposition}{proposition}
\crefalias{lemma}{lemma}
\crefalias{corollary}{corollary}
\crefname{definition}{Definition}{Definitions}
\crefname{proposition}{Proposition}{Propositions}
\crefname{lemma}{Lemma}{Lemmas}
\crefname{corollary}{Corollary}{Corollaries}

\newcommand{\Ric}{\mathrm{Ric}}
\newcommand{\Scal}{\mathrm{Scal}}
\newcommand{\Vol}{\mathrm{Vol}}
\newcommand{\tr}{\mathrm{tr}}
\newcommand{\Hess}{\mathrm{Hess}}
\renewcommand{\d}{\,\mathrm d}
\newcommand{\Rk}{\mathsf{R}}
\newcommand{\GF}{\mathsf{G}^{\mathrm F}}
\newcommand{\DF}{\mathsf{D}^{\mathrm F}}
\newcommand{\E}{\mathsf{E}}
\newcommand{\Q}{\mathsf{Q}}
\newcommand{\R}{\mathbb R}

\title{\LARGE\bfseries
Heat Codimension and Scalar Curvature}
\author{Bang-Xian Han\thanks{School of Mathematics, Shandong University,
Jinan 250100, China. Email: hanbx@sdu.edu.cn}}
\date{}

\begin{document}
\maketitle

\begin{abstract}
Recent proofs of Gromov's codimension-two volume-growth conjecture suggest a
heat-entropy formulation of scalar curvature.  We propose a synthetic
scalar lower bound on metric measure Dirichlet spaces with a  heat
codimension parameter.  The condition recovers scalar curvature in the smooth
setting, is stable under tensorzation and entropy--heat convergence, and gives
dimension bounds for blow-downs.
\end{abstract}



\section{Introduction}

Gromov asked whether the volume-growth estimate that he stated under
nonnegative sectional curvature remains valid under nonnegative Ricci
curvature: if a complete noncompact $N$-manifold satisfies
$\Ric\ge0$ and $\Scal\ge\sigma>0$, must
\[
 \Vol B_R(x)\le C(N)\sigma^{-1}R^{N-2}\qquad(R>0)?
\]
See \cite[p.~114, (5') and \S2.A(b)]{GromovLarge}.  This question has
recently been answered affirmatively by three different methods
\cite{HeatKernel2026,KongZhu2026,Antonelli2026}; the last reference treats
positive intermediate curvature as well.  Kong--Zhu also prove a
codimension-one estimate starting from a uniform deficit in the volumes of
unit balls \cite{KongZhu2026}.  That assumption is macroscopic and is not part
of the scalar condition considered here.

We formulate a heat-entropy condition suggested by the codimension-two
estimate.  As in the curvature-dimension theory
\cite{LottVillani,SturmI,SturmII}, curvature is expressed by finite-scale
concavity; here the curves lie in the heat-time plane rather than in the base
space or in Wasserstein space.  Other weak notions of scalar lower bounds
include distributional scalar curvature \cite{LeeLeFloch}, Ricci-flow
regularization for $C^0$ metrics \cite{BurkhardtGuim}, and Gromov's
macroscopic scalar curvature based on finite-radius ball volumes
\cite{GromovFour,Sabourau}.

The two recent large-time estimates used below are \cite[Theorem~3.1]{HeatKernel2026}
for positive scalar curvature and \cite[Theorem~3.1]{Koirala2026} for
positive intermediate curvature.   In the
noncompact smooth setting we also use the positive-time heat-kernel identities
and regularity established in \cite[Sections~2.4--2.5 and Appendix~B]{HeatKernel2026}.

Let $(X,\mathsf d,\mathfrak m,\mathcal E)$ be a metric measure Dirichlet
space with heat semigroup $P_t$, heat kernel $p_t$, and
$\mu_{t,x}:=p_t(x,\cdot)\mathfrak m$.  For $N>0$ put
\begin{align*}
 \operatorname{Ent}_{\mathfrak m}(\mu_{t,x})
 &:=\int_X p_t(x,y)\log p_t(x,y)\,\d\mathfrak m(y),\\
 \mathcal N_x^{(N)}(t)
 &:=-\operatorname{Ent}_{\mathfrak m}(\mu_{t,x})
   -\frac N2\log(4\pi t)-\frac N2,\\
 \mathbf N_x^{(N)}(t,s)
 &:=P_s\bigl(\mathcal N_\bullet^{(N)}(t)\bigr)(x).
\end{align*}
The second heat time is needed because, in the smooth case, differentiation
along $t+s=T$ produces the heat average $P_s\Q_t$ of the scalar source; see
Proposition~\ref{prop:smooth-identity}.

The logarithmic correction is fixed by Euclidean space.  If $\mathbb R^m$ is
normalized by $N\ge m$, then
$\mathcal N_x^{(N)}(t)=\mathrm{const}-(N-m)\log t/2$, so in inverse time
$r=t^{-1}$ its second derivative is $-(N-m)/(2r^2)$.

Set $\mathscr T=(0,\infty)^2$ and $\pi(t,s)=t+s$.  On $\mathscr T$ consider
the extended metric
\[
 \mathsf d_{\mathrm{it}}\bigl((t,s),(\bar t,\bar s)\bigr)
 =\begin{cases}
 |t^{-1}-\bar t^{-1}|,&t+s=\bar t+\bar s,\\
 +\infty,&t+s\ne\bar t+\bar s.
 \end{cases}
\]
Its finite-distance components are $\Sigma_T:=\pi^{-1}(T)$, and
$(t,s)\mapsto t^{-1}$ is an isometry from $\Sigma_T$ onto
$(T^{-1},\infty)$.  Thus the finite-distance constant-speed geodesics are the
curves $\gamma_\theta=(t_\theta,s_\theta)$ with $t_\theta+s_\theta$ constant
and
$t_\theta^{-1}=(1-\theta)t_0^{-1}+\theta t_1^{-1}$.  We call these
\emph{inverse-time geodesics}.  Write
$\mathscr T^+:=\{(t,s):t\le s\}$ and
$\mathscr T_S^+:=\{(t,s):S\le t\le s\}$.

\begin{definition}[Heat scalar lower bound]\label{def:HScal}
Let $\sigma\in\mathbb R$ and $N>0$.  We say that $X$ has \emph{heat scalar
curvature at least $\sigma$ in dimension $N$} if there is a scale $\tau>0$,
independent of the pole $x$, such that
$\mathbf N_x^{(N)}(t,s)+(\sigma/2)t$ is concave along every inverse-time
geodesic contained in $(0,\tau)^2$, for every $x\in X$.
\end{definition}

For $\sigma=0$ this is exactly local inverse-time concavity of
$\mathbf N_x^{(N)}$.

\begin{definition}[Heat scalar--codimension condition]
Let $\sigma\in\mathbb R$, $N>0$, and $q\in[0,N]$.  We say that
$X\in\mathrm{HSCod}(\sigma,N,q)$ if it has heat scalar curvature at least
$\sigma$ in dimension $N$ and, when $q>0$, there is a scale $S>0$,
independent of the pole, such that
$\mathbf N_x^{(N)}(t,s)+(q/2)\log t$ is concave along every inverse-time
geodesic contained in $\mathscr T_S^+$, for every $x\in X$.  For $q=0$
there is no large-time condition and this is Definition~\ref{def:HScal}.
\end{definition}

\begin{remark}[Heat codimension]\label{rem:heat-codimension}
The name refers to the large-time heat profile, not to a local codimension of
$X$.  With normalization dimension $N$, Euclidean space $\mathbb R^m$ has
sharp coefficient $N-m$; see Proposition~\ref{prop:euclidean}.  More
generally, if $K^k$ is closed and connected, heat-kernel factorization and
\[
 p_t^K(x,y)=\Vol(K)^{-1}+\sum_{j\ge1}e^{-\lambda_jt}\phi_j(x)\phi_j(y)
\]
give a sharp large-time coefficient $k$ on
$K^k\times\mathbb R^{N-k}$.  Thus
\[
\begin{array}{c|c}
\text{space} & \text{sharp large-time coefficient}\\ \hline
\mathbb R^N & 0\\
K^k\times\mathbb R^{N-k} & k\\
K^N\ \text{closed} & N.
\end{array}
\]
The estimate $q<2$ in \cref{thm:positive-scalar} is therefore a uniform
lower bound; it need not be sharp for a fixed space.  Relating the sharp
coefficient to the geometry of blow-downs is open.
\end{remark}

\begin{theorem}[Smooth scalar curvature]\label{thm:smooth-identification}
Let $(M^N,g)$ be closed.  If $M$ has heat scalar curvature at least $\sigma$,
then $\Scal_g\ge\sigma$.  Conversely, if $\min_M\Scal_g>\sigma$, then $M$
has heat scalar curvature at least $\sigma$.  Hence
\[
 \min_M\Scal_g
 =\sup\{\sigma:M\text{ has heat scalar curvature at least }\sigma\}.
\]
\end{theorem}

\begin{theorem}[Positive scalar curvature]\label{thm:positive-scalar}
Let $(M^N,g)$ be complete with $\Ric_g\ge0$ and
$\Scal_g\ge\sigma_0>0$.  Then $M$ has heat scalar curvature at least every
$\sigma<\sigma_0$.  If, in addition, $M$ is connected and noncompact and
$N\ge3$, then
$M\in\mathrm{HSCod}(\sigma,N,q)$ for every
$\sigma<\sigma_0$ and $0<q<2$.
\end{theorem}

The definitions are monotone in $\sigma$ and $q$, have the scalar-curvature
scaling law, are unchanged by constant renormalization of the reference
measure, and satisfy the product rules stated in Proposition~\ref{prop:basic-properties}.
They are also stable under entropy--heat convergence when the witness scales
are uniform; see \cref{thm:stability}.

\begin{theorem}[Blow-down comparison]\label{thm:blowdown-dimension}
Let $X\in\mathrm{HSCod}(\sigma,N,q)$ with $q>0$, and suppose that
$(X,a_i\mathsf d,c_i\mathfrak m,x_i)$ converges in entropy--heat sense to
$(X_\infty,\mathsf d_\infty,\mathfrak m_\infty,x_\infty)$ for some
$a_i\downarrow0$.  Then $X_\infty$ satisfies the exact forward
codimension-$q$ comparison.  If a further pointed rescaling of $X_\infty$
converges in entropy--heat sense to
$(\mathbb R^m,\mathsf d_{\mathrm E},c\mathcal L^m,0)$, then $m\le N-q$.
\end{theorem}

For a finite-dimensional $\mathrm{RCD}$ space, the \emph{essential
dimension} is its almost-everywhere Euclidean regular dimension
\cite{MondinoNaber,BrueSemola}.  Combining \cref{thm:positive-scalar} with
finite-dimensional $\mathrm{RCD}$ compactness and stability
\cite{SturmII,EKS} and the theorem above gives the
codimension-two dimension bound for normalized blow-downs of complete manifolds with
$\Ric\ge0$ and a uniform positive scalar lower bound; see
Corollary~\ref{cor:smooth-blowdown}.  The codimension-one theorem of \cite{KongZhu2026} starts instead from a
unit-ball volume deficit and is not covered by the scalar condition here.

\section{Smooth heat identities}

Let $(M^n,g)$ be a complete stochastically complete Riemannian manifold with
minimal heat kernel $p_t(x,y)$.  Denote the heat semigroup on functions by
$P_t$ and its dual heat flow on measures by $\mathsf H_t=(P_t)^*$.  Thus $\mu_{t,x}:=\mathsf H_t\delta_x$ has density
$\d\mu_{t,x}(y)=p_t(x,y)\,\d\Vol_g(y)$.

\subsection{The normalized Fisher heat tensor}

For $v\in T_xM$ define the score
$\ell_{t,x,v}(y):=\mathrm d_x\log p_t(x,y)[v]$.
Stochastic completeness and differentiation in the pole variable give
\[
 \int_M\ell_{t,x,v}\,\d\mu_{t,x}=0,
 \qquad
 \mathrm d(P_t f)_x(v)=\int_M f\ell_{t,x,v}\,\d\mu_{t,x}.
\]

\begin{definition}[Normalized Fisher heat tensor]
The normalized Fisher pullback tensor is
\[
 \GF_{t,x}(v,w)
 :=2t\int_M\ell_{t,x,v}\ell_{t,x,w}\,\d\mu_{t,x}.
\]
Its defect is $\DF_t:=g-\GF_t$.
\end{definition}

On $\R^n$ one has $\GF_t=g$ at every positive time.  Under
$\Ric\ge0$, the reverse local Poincar\'e inequality
\[
 P_t(f^2)-(P_t f)^2\ge2t|\nabla P_t f|^2.
\]
holds, and duality in $L^2(\mu_{t,x})$ yields $0\le\GF_t\le g$; see
\cite[Theorem~4.7.2(iv)]{BGL}.

\subsection{The centered source and its models}

\subsubsection{The trace-level source}

Write $h=-\log p_t$.  Associated with the Fisher tensor is the
heat-dimension defect
$\E_t(x):=\frac t2\bigl(n-\tr_g\GF_{t,x}\bigr)$.
The corresponding centered source is
\begin{align}
 \Q_t(x)
 :={}&4t^2\int_M\left|
 \Hess_xh(x,y,t)-\frac{\GF_{t,x}}{2t}
 \right|^2\d\mu_{t,x}(y)\notag\\
 &+|g_x-\GF_{t,x}|^2
 +2t\langle\Ric_x,\GF_{t,x}\rangle.
 \label{eq:Q}
\end{align}
The Hessian--Ricci production term in \eqref{eq:Q} is the one appearing in
the linear heat entropy formula \cite[Theorem~0.1]{Ni}.  The pole-variable
centered identity and the positive-time regularity used in the complete
noncompact case are established in \cite[Sections~2.4--2.5 and
Appendix~B]{HeatKernel2026}.  Standard heat-kernel asymptotics give
\[
 \Q_t(x)=2t\Scal(x)+O_x(t^2),\qquad t\downarrow0.
\]
See \cite[Chapter~2]{BGV} for the local heat-kernel expansion and
\cite[Section~2]{HeatKernel2026} for the normalization used here.

Let $a_1(x)\le\cdots\le a_n(x)$ be the eigenvalues of
$\Ric_x^\sharp$, and set $\Rk_k(x):=a_1(x)+\cdots+a_k(x)$ for
$1\le k\le n$.  Thus $\Rk_n=\Scal$.  We also use
\[
 \Psi(u):=1-(1-u)_+^2
 =\begin{cases}2u-u^2,&0\le u\le1,\\1,&u\ge1.\end{cases}
\]

\begin{lemma}[Scalar source floor]\label{lem:scalar-source-floor}
Let $(M^n,g)$ be complete with $\Ric\ge0$ and $\Scal\ge\sigma>0$.
Then for every $x\in M$ and $t>0$,
\begin{equation}\label{eq:scalar-source-floor}
 \Q_t(x)\ge\Psi(\sigma t).
\end{equation}
\end{lemma}

\begin{proof}
Fix $x,t$ and set $A=\Ric_x^\sharp\ge0$ and
$G=(\GF_{t,x})^\sharp$.  By $0\le G\le I$, if $\lambda$ is the
smallest eigenvalue of $G$, then
\[
 \langle A,G\rangle\ge\lambda\,\tr A\ge\sigma\lambda,
 \qquad |I-G|^2\ge(1-\lambda)^2.
\]
Discarding the nonnegative Hessian-variance term in \eqref{eq:Q} gives
\[
 \Q_t(x)\ge(1-\lambda)^2+2\sigma t\lambda.
\]
Minimizing over $0\le\lambda\le1$ gives $\Psi(\sigma t)$ and proves \eqref{eq:scalar-source-floor}.
\end{proof}

\subsection{Curvature estimates}

\begin{theorem}[Partial Ricci curvature]\label{thm:partial-ricci}
Let $(M^n,g)$ be complete with $\Ric\ge0$.  Suppose that
$2\le k<n$ and $\Rk_k\ge\sigma>0$, and set $q=n-k+1$.  Then
\[
 \Q_t(x)\ge q\,\Psi\!\left(\frac{\sigma t}{k}\right),
 \qquad
 \Scal(x)\ge\frac{n\sigma}{k}.
\]
Consequently, for every $\sigma'<n\sigma/k$,
\[
 M\in\mathrm{HSCod}(\sigma',n,n-k+1),
\]
and the exact forward codimension-$q$ comparison already holds on
$\mathscr T_{k/\sigma}^+$.  At the scalar endpoint $k=n$, the pointwise estimate gives one logarithmic
unit, while the propagated scalar estimate below gives the endpoint two.
\end{theorem}

\begin{proof}
Fix $x\in M$ and set $A=\Ric_x^\sharp$ and
$G=(\GF_{t,x})^\sharp$.  Let
$\alpha_1\le\cdots\le\alpha_n$ be the eigenvalues of $A$.
Since $\alpha_1+\cdots+\alpha_k\ge\sigma$, one has
$\alpha_k\ge\sigma/k$.  If $P$ is the spectral projection onto the
eigenspaces corresponding to $\alpha_k,\ldots,\alpha_n$, then
$\operatorname{rank}P=q$ and $A\ge(\sigma/k)P$.
Let $0\le\lambda_1\le\cdots\le\lambda_n\le1$ be the eigenvalues of $G$.
The minimum principle for partial eigenvalue sums gives
\[
 \langle A,G\rangle\ge\frac\sigma k\tr(PG)
 \ge\frac\sigma k\sum_{i=1}^q\lambda_i.
\]
Discarding the first, nonnegative term in \eqref{eq:Q},
\begin{align*}
 \Q_t(x)
 &\ge |I-G|^2+2t\langle A,G\rangle\\
 &\ge\sum_{i=1}^q
 \left((1-\lambda_i)^2+2\frac{\sigma t}{k}\lambda_i\right)\\
 &\ge q\,\Psi\!\left(\frac{\sigma t}{k}\right),
\end{align*}
because the minimum of $(1-\lambda)^2+2u\lambda$ on $[0,1]$ is $\Psi(u)$.
Also
$\Scal=\tr A\ge\sigma+(n-k)\alpha_k\ge n\sigma/k$.
For $t\ge k/\sigma$ the source bound becomes $\Q_t\ge q$; the smooth
identity in Proposition~\ref{prop:smooth-identity} gives the exact
forward codimension-$q$ comparison.  For every $\sigma' < n\sigma/k$, the microscopic scalar clause follows from
part~(i) of \cref{thm:positive-scalar}.
\end{proof}

For an orthonormal frame $e_1,\ldots,e_n$, define
\begin{equation}\label{eq:BHJ-intermediate-curvature}
 \mathcal C_m(e_1,\ldots,e_m)
 :=\sum_{p=1}^m\sum_{r=p+1}^n
   \operatorname{Rm}(e_p,e_r,e_p,e_r),
 \qquad 1\le m\le n-1.
\end{equation}
Thus $\mathcal C_1(e_1)=\Ric(e_1,e_1)$ and
$\mathcal C_{n-1}=\Scal/2$ \cite[Definition~1.1]{BHJ}.

\begin{theorem}[Positive intermediate curvature]
\label{thm:intermediate-curvature}
Let $(M^n,g)$ be complete and connected with $\Ric\ge0$, and suppose that
for every orthonormal frame $e_1,\ldots,e_n$,
\[
 \mathcal C_m(e_1,\ldots,e_m)\ge\kappa>0
 \qquad(1\le m\le n-1).
\]
Set
\[
 q:=n-m+1,\qquad
 \sigma_m:=\frac{2n(n-1)}{m(2n-m-1)}\,\kappa.
\]
Then, for all sufficiently large $\kappa S$,
\begin{equation}
 \inf_{\substack{x\in M\\s\ge t\ge S}}
 P_s\Q_t(x)
 \ge q-C_{n,m}(\kappa S)^{-1/m}.
\end{equation}
In particular, for every $\sigma'<\sigma_m$ and every $q'<q$,
\[
 M\in\mathrm{HSCod}(\sigma',n,q').
\]
At the scalar endpoint $m=n-1$, the sharp endpoint is therefore $q=2$ in
the sense that every $q'<2$ is realized at a finite forward scale.
\end{theorem}

\begin{proof}
Let $0\le\lambda_1(x,t)\le\cdots\le\lambda_n(x,t)\le1$ be the eigenvalues
of $(\GF_{t,x})^\sharp$.  The heat-averaged eigenvalue estimate
\cite[Theorem~3.1]{Koirala2026} gives
\[
 P_t\lambda_q(\cdot,t)(x)
 \le C_{n,m}\min\{1,(\kappa t)^{-1/m}\}.
\]
If $s\ge t$, the semigroup property therefore yields
\[
 P_s\lambda_q(\cdot,t)(x)
 \le C_{n,m}(\kappa t)^{-1/m}
\]
for large $\kappa t$.  Since $\Ric\ge0$, \eqref{eq:Q} gives
\[
 \Q_t\ge |g-\GF_t|^2\ge q(1-\lambda_q)^2.
\]
Hence
\[
 P_s\Q_t(x)
 \ge q\bigl(1-P_s\lambda_q(\cdot,t)(x)\bigr)^2
 \ge q-C_{n,m}(\kappa t)^{-1/m},
\]
after enlarging $C_{n,m}$; factors depending on $q\le n$ are absorbed into
the dimensional constant.
The smooth identity converts this source lower bound on
$s\ge t\ge S$ into the corresponding forward codimension comparison.

Averaging \eqref{eq:BHJ-intermediate-curvature} over orthonormal frames gives
\[
 \mathbb E_{\mathrm{frame}}[\mathcal C_m]
 =\frac{m(2n-m-1)}{2n(n-1)}\Scal.
\]
Thus $\mathcal C_m\ge\kappa$ implies $\Scal\ge\sigma_m$, and the microscopic
scalar clause follows from part~(i) of \cref{thm:positive-scalar}.
\end{proof}

For $m=n-1$, one has $\mathcal C_{n-1}=\Scal/2$ and $q=2$.  The product
$S^2_a\times\mathbb R^{n-2}$, with $a^2=2/\sigma$, realizes this endpoint.
For $m=1$ the hypothesis is $\Ric\ge\kappa>0$, so Bonnet--Myers forces compactness; the noncompact intermediate-curvature applications begin with $m=2$.

\section{Synthetic formulation}

The definitions below use only the heat kernel, its entropy, and the
semigroup.

\subsection{Metric measure Dirichlet spaces and pointed Nash entropy}

\begin{definition}[Metric measure Dirichlet space]
A \emph{metric measure Dirichlet space} (MMD space) consists of a proper
complete separable metric space $(X,\mathsf d)$, a full-support Radon measure
$\mathfrak m$, and a strongly local regular symmetric Dirichlet form
$(\mathcal E,\mathcal D(\mathcal E))$ on $L^2(X,\mathfrak m)$.  We assume
that its energy measures are absolutely continuous with respect to
$\mathfrak m$ and write $\Gamma$ for the carr\'e du champ density.  The
intrinsic distance
\[
 \mathsf d_{\mathcal E}(x,y)
 :=\sup\bigl\{f(x)-f(y):f\in\mathcal D_{\mathrm{loc}}(\mathcal E)\cap C(X),
                  \ \Gamma(f)\le1\ \mathfrak m\text{-a.e.}\bigr\}
\]
is required to equal $\mathsf d$.
\end{definition}

The terminology is standard in the Dirichlet-form literature
\cite{KajinoMurugan}; compare also the Riemannian Energy measure spaces of
\cite{AGSBE}.

Throughout the paper we impose the following standing heat assumptions on an
MMD space.  The semigroup $(P_t)_{t>0}$ is conservative and admits a jointly
continuous, strictly positive, symmetric heat kernel
$p:(0,\infty)\times X\times X\to(0,\infty)$ such that, for bounded
Borel $f$ and all $s,t>0$,
\begin{align*}
 P_t f(x)&=\int_Xp_t(x,y)f(y)\,\d\mathfrak m(y),
 &\int_Xp_t(x,y)\,\d\mathfrak m(y)&=1,\\
 p_{t+s}(x,y)&=\int_Xp_t(x,z)p_s(z,y)\,\d\mathfrak m(z).
\end{align*}
For the normalization dimensions used below, we assume that the pointed
Nash entropies are finite for every pole and every positive time, and that
$\mathcal N_\bullet^{(N)}(t)$ is integrable against
$p_s(x,\cdot)\mathfrak m$ for every $x$ and $s,t>0$.

We reserve $P_t$ for the heat semigroup on functions.  Its dual heat
flow on finite measures is denoted by $\mathsf H_t=(P_t)^*$ and is
defined by
$\int_X f\,\d(\mathsf H_t\mu)=\int_XP_t f\,\d\mu$ for bounded Borel
$f$.  In particular,
$\mu_{t,x}:=\mathsf H_t\delta_x=p_t(x,\cdot)\mathfrak m$.
The smooth manifolds used in Section~2 satisfy these assumptions under the
hypotheses imposed there; in the complete noncompact case the positive-time
statements used below are collected in \cite[Sections~2.4--2.5 and
Appendix~B]{HeatKernel2026}.  Finite-dimensional $\mathrm{RCD}(0,N)$ spaces
have the conservative heat flow and heat-kernel bounds used below; see
\cite{EKS,JiangLiZhang}.

For $\mu=\rho\mathfrak m$ define
$\operatorname{Ent}_{\mathfrak m}(\mu):=\int_X\rho\log\rho\,\d\mathfrak m$.
Fix a normalization dimension $N>0$.  The $N$-normalized pointed Nash
entropy is
\begin{equation}\label{eq:Nash-entropy}
 \mathcal N_x^{(N)}(t)
 :=-\operatorname{Ent}_{\mathfrak m}(\mu_{t,x})
 -\frac N2\log(4\pi t)-\frac N2.
\end{equation}
For $t,s>0$ set
\begin{equation}\label{eq:propagated-Nash}
 \mathbf N_x^{(N)}(t,s)
 :=P_s\bigl(\mathcal N_\bullet^{(N)}(t)\bigr)(x).
\end{equation}
The finite-distance components of $\mathsf d_{\mathrm{it}}$ are the
heat-time leaves $\Sigma_T=\pi^{-1}(T)$, identified isometrically with
$(T^{-1},\infty)$ by $r=t^{-1}$.  The inverse-time Nash profile is the
coordinate expression of \eqref{eq:propagated-Nash} on such a component:
\[
 \mathcal H_{T,x}(r)
 :=\mathbf N_x^{(N)}(r^{-1},T-r^{-1}),
 \qquad r>T^{-1}.
\]

\subsection{Inverse-time concavity and smooth calibration}

For one heat-time leaf $\Sigma_T$, let $T^{-1}<a<b<c$ and put
$\lambda=(b-a)/(c-a)$.  The three-point forms of the two model corrections
are
\begin{align*}
 \mathfrak d^0(a,b,c)
 &:=\frac12\left(\frac{1-\lambda}{a}+\frac\lambda c-\frac1b\right),\\
 \mathfrak d^\infty(a,b,c)
 &:=\frac12\left(\log b-(1-\lambda)\log a-\lambda\log c\right).
\end{align*}
Since $r=t^{-1}$ is an isometric coordinate, the scalar lower bound $\sigma$
with witness scale $\tau$ is equivalent to
\[
 \mathcal H_{T,x}(b)
 \ge(1-\lambda)\mathcal H_{T,x}(a)+\lambda\mathcal H_{T,x}(c)
   +\sigma\mathfrak d^0(a,b,c)
\]
for every triple whose heat-time points lie in $(0,\tau)^2$.  For $q>0$,
the macroscopic clause with witness scale $S$ is equivalent to
\[
 \mathcal H_{T,x}(b)
 \ge(1-\lambda)\mathcal H_{T,x}(a)+\lambda\mathcal H_{T,x}(c)
   +q\mathfrak d^\infty(a,b,c)
\]
for every triple in $\mathscr T_S^+$.

We say that $X$ satisfies the \emph{exact forward codimension-$q$
comparison} if $\mathbf N_x^{(N)}+(q/2)\log t$ is concave along every
inverse-time geodesic contained in $\mathscr T^+$.  Equivalently, on every
forward half-leaf,
\begin{equation}
 \mathcal H_{T,x}''(r)\le-\frac{q}{2r^2}
\end{equation}
in the distributional sense.

\begin{proposition}[Smooth identity]\label{prop:smooth-identity}
Let $(M^n,g)$ be smooth and use the Nash normalization $N=n$.  Assume that
the positive-time heat quantities above are finite and smooth and that the
pole differentiations may be passed through the heat-kernel integrals.  This
holds, in particular, on closed manifolds and in the complete
nonnegative-Ricci setting used below.  Then
\begin{equation}\label{eq:H-Q}
 \mathcal H_{T,x}''(r)
 =-\frac{1}{2r^2}
 P_{T-r^{-1}}\Q_{r^{-1}}(x).
\end{equation}
Consequently, for every $\tau,S>0$:
\begin{enumerate}[label=\textup{(\roman*)}]
\item $\mathbf N_x^{(n)}+(\sigma/2)t$ is concave along every inverse-time
geodesic in $(0,\tau)^2$ if and only if
\begin{equation}\label{eq:smooth-micro-source}
 P_s\Q_t(x)\ge2\sigma t
 \qquad(x\in M,\ 0<t,s<\tau).
\end{equation}
\item $\mathbf N_x^{(n)}+(q/2)\log t$ is concave along every inverse-time
geodesic in $\mathscr T_S^+$ if and only if
\begin{equation}\label{eq:smooth-macro-source}
 P_s\Q_t(x)\ge q
 \qquad(x\in M,\ s\ge t\ge S).
\end{equation}
\end{enumerate}
\end{proposition}

\begin{proof}
The pointed Nash-entropy identities give
\[
 (\partial_t-\Delta)\mathcal N_\bullet^{(n)}(t)=-\frac{\E_t}{t^2},
 \qquad
 (\partial_t-\Delta)\E_t=\frac12\Q_t.
\]
Differentiating \eqref{eq:propagated-Nash} along $t+s=T$ yields
\[
 \frac{\d}{\d t}\left(t^2\frac{\d}{\d t}
 \mathbf N_x^{(n)}(t,T-t)\right)
 =-\frac12P_{T-t}\Q_t(x).
\]
The substitution $r=t^{-1}$ gives \eqref{eq:H-Q}.  Adding
$\sigma/(2r)$ or $-(q/2)\log r$ to $\mathcal H_{T,x}$ and differentiating
twice gives \eqref{eq:smooth-micro-source} and
\eqref{eq:smooth-macro-source}.
\end{proof}

\begin{remark}[Normalization dimension]
The identity above is written with the geometric normalization $N=n$.  If the
same $n$-manifold is normalized instead by $\bar N\ne n$, then
\[
 \mathcal H_{T,x}^{(\bar N)}(r)
 =\mathcal H_{T,x}^{(n)}(r)+\frac{\bar N-n}{2}\log r+\text{constant},
\]
and hence
\[
 (\mathcal H_{T,x}^{(\bar N)})''(r)
 =-\frac{1}{2r^2}\Bigl(P_{T-r^{-1}}\Q_{r^{-1}}(x)
       +\bar N-n\Bigr).
\]
The additional term is the normalization-dimension contribution that also
appears in the Euclidean calculation below.  It also shows why the geometric
normalization is singled out by scalar identification.  As $r\to\infty$,
$P_{T-r^{-1}}\Q_{r^{-1}}=O(r^{-1})$.  Hence if $\bar N>n$, the negative
term $-(\bar N-n)/(2r^2)$ dominates and the microscopic condition holds for
every finite scalar coefficient at sufficiently small heat time; if
$\bar N<n$, the opposite sign prevents any finite scalar coefficient from
holding at sufficiently small heat time.  Only $\bar N=n$ gives the
nondegenerate scalar limit of \cref{thm:smooth-identification}.
\end{remark}

\begin{proof}[Proof of Theorem~\ref{thm:smooth-identification}]
Uniformly as $t,s\downarrow0$,
\[
 \frac{P_s\Q_t(x)}{2t}=\Scal(x)+o(1).
\]
If the finite-scale scalar comparison with coefficient $\sigma$ holds,
Proposition~\ref{prop:smooth-identity} gives
$P_s\Q_t(x)/(2t)\ge\sigma$ on some small heat square; letting
$t,s\downarrow0$ yields $\Scal\ge\sigma$.  Conversely, if
$\min_M\Scal>\sigma$, uniformity of the same expansion gives a witness
$\tau>0$ on which $P_s\Q_t/(2t)\ge\sigma$.  The smooth identity
then gives the required inverse-time concavity.  Formula
The equality in Theorem~\ref{thm:smooth-identification} follows immediately.  The supremum need not be attained, because the finite witness scale in the converse may shrink to zero at the endpoint.
\end{proof}

\begin{proof}[Proof of Theorem~\ref{thm:positive-scalar}]
For (i), Lemma~\ref{lem:scalar-source-floor} gives
$\Q_t\ge\Psi(\sigma_0t)$.  Hence for $0<t\le\sigma_0^{-1}$,
\[
 \frac{P_s\Q_t}{2t}\ge\sigma_0-\frac{\sigma_0^2t}{2},
\]
and the stated choice of $\tau$ makes the right-hand side at least $\sigma$.
The conclusion follows from Proposition~\ref{prop:smooth-identity}.

For (ii), after scaling the scalar lower bound to one, the propagated source
estimate \cite[Theorem~3.1]{HeatKernel2026} gives
\[
 P_s\Q_t(x)
 \ge2-C_N\bigl((\sigma_0t)^{-1}+(\sigma_0s)^{-1}\bigr)^{1/3}.
\]
After enlarging $C_N$, the condition $C_N(\sigma_0S)^{-1/3}\le2-q$ implies $P_s\Q_t\ge q$ for $s\ge t\ge S$.  The macroscopic comparison follows
from Proposition~\ref{prop:smooth-identity}, and the microscopic clause follows from
(i).
\end{proof}

\subsection{Basic properties}

\begin{proposition}[Basic properties]\label{prop:basic-properties}
It holds the following properties:
\begin{enumerate}[label=\textup{(\roman*)}]
\item \emph{Monotonicity.}  If $X\in\mathrm{HSCod}(\sigma,N,q)$, then
$X\in\mathrm{HSCod}(\sigma',N,q')$ for every
$\sigma'\le\sigma$ and $0\le q'\le q$, with the same witness scales whenever
the corresponding clause is nonvacuous.
\item \emph{Scaling and measure renormalization.}  For $a,c>0$, equip $X$
with $\mathsf d_a=a\mathsf d$, $\mathfrak m_a=c\mathfrak m$, and
$\mathcal E_{a,c}=ca^{-2}\mathcal E$.  Then
\[
 X\in\mathrm{HSCod}(\sigma,N,q)
 \quad\Longrightarrow\quad
 (X,a\mathsf d,c\mathfrak m)
 \in\mathrm{HSCod}(a^{-2}\sigma,N,q).
\]
The same statement with only the scalar clause also holds.
\item \emph{Products.}  If $X_i$ has heat scalar curvature at least
$\sigma_i$ in dimension $N_i$, $i=1,2$, then $X_1\times X_2$ has heat scalar
curvature at least $\sigma_1+\sigma_2$ in dimension $N_1+N_2$.  If moreover
$q_1,q_2>0$ and
$X_i\in\mathrm{HSCod}(\sigma_i,N_i,q_i)$, then
\[
 X_1\times X_2
 \in\mathrm{HSCod}(\sigma_1+\sigma_2,N_1+N_2,q_1+q_2).
\]
If $X\in\mathrm{HSCod}(\sigma,N,q)$, then for every $m\ge1$,
\[
 X\times\mathbb R^m\in\mathrm{HSCod}(\sigma,N+m,q),
\]
with the Euclidean factor normalized by its dimension.
\end{enumerate}
\end{proposition}

\begin{proof}
For (i), along an inverse-time geodesic both $t$ and $\log t$ are
convex functions of the affine coordinate $r=t^{-1}$.  Decreasing either
coefficient therefore adds a concave function to the corresponding corrected
Nash entropy.

For (ii), the rescaled semigroup is $P_t^a=P_{a^{-2}t}$ and
\[
 \mathcal N_x^{(N),(a,c)}(t)
 =\mathcal N_x^{(N)}(a^{-2}t)-N\log a+\log c.
\]
After the common quadratic time rescaling, the scalar coefficient becomes
$a^{-2}\sigma$, while the logarithmic correction and the forward sector are
unchanged.

For (iii), product heat kernels factorize, so normalized pointed Nash entropy
and its propagated two-time functional are additive.  Taking the smaller
microscopic witness scale and, when $q_1,q_2>0$, the larger forward witness
scale gives the product assertions.  For $\mathbb R^m$ normalized by $m$,
the normalized pointed Nash entropy is identically zero.
\end{proof}

\subsection{Entropy--heat stability}

\begin{definition}[Entropy--heat convergence]
A sequence of pointed MMD spaces
$(X_i,\mathsf d_i,\mathfrak m_i,\mathcal E_i,x_i)$ converges in
\emph{entropy--heat sense} to
$(X,\mathsf d,\mathfrak m,\mathcal E,x)$ if the underlying pointed measured
spaces converge in a fixed geometric construction and, for every convergent
sequence of poles $z_i\to z$,
\[
 \mathcal N_{z_i}^{(N),i}(t)\longrightarrow\mathcal N_z^{(N)}(t),
 \qquad
 \mathbf N_{z_i}^{(N),i}(t,s)\longrightarrow\mathbf N_z^{(N)}(t,s)
\]
locally uniformly for $t>0$ and for $(t,s)\in(0,\infty)^2$, respectively.
\end{definition}

Equivalently, $\mathcal H_{T,z_i}^i$ converges locally uniformly on each
fixed heat-time leaf.

\begin{proposition}[Pointed $\mathrm{RCD}$ construction of entropy--heat
convergence]\label{prop:RCD-entropy-heat}
Let $1<N<\infty$ and let
$(X_i,\mathsf d_i,\mathfrak m_i,x_i)$ be pointed $\mathrm{RCD}(0,N)$ spaces with
\begin{equation}\label{eq:RCD-normalization}
 \mathfrak m_i(B_1(x_i))=1.
\end{equation}
If they converge in the pointed measured Gromov--Hausdorff sense to
$(X,\mathsf d,\mathfrak m,x)$, then they converge in entropy--heat sense.  More
specifically, for convergent poles $z_i\to z$,
\begin{equation}\label{eq:RCD-entropy-heat-local}
 \mathcal N_{z_i}^{(N),i}(t)\longrightarrow\mathcal N_z^{(N)}(t),
 \qquad
 P^i_s\bigl(\mathcal N_\bullet^{(N),i}(t)\bigr)(z_i)
 \longrightarrow
 P_s\bigl(\mathcal N_\bullet^{(N)}(t)\bigr)(z)
\end{equation}
locally uniformly for $t>0$ and for $(t,s)\in(0,\infty)^2$, respectively.
\end{proposition}

\begin{proof}
Use a common construction of the pointed measured convergence.  Positive-time
heat kernels satisfy
\begin{equation}\label{eq:RCD-kernel-pointwise-convergence}
 p^i_{t_i}(z_i,y_i)\longrightarrow p_t(z,y)
 \quad\text{whenever}\quad
 (z_i,y_i,t_i)\longrightarrow(z,y,t),\quad t>0,
\end{equation}
by \cite[Theorem~3.3]{AHT}; the local Lipschitz estimate
\cite[(3.4)]{AHT} makes the convergence locally uniform on bounded pole
cylinders with time in $I\Subset(0,\infty)$.

Fix $D<\infty$ and
$I=[t_0,t_1]\Subset(0,\infty)$, and let $z_i\in B_D(x_i)$.  Bishop--Gromov comparison for $\mathrm{RCD}(0,N)$ spaces \cite{SturmII,EKS} and
\eqref{eq:RCD-normalization} give, uniformly for $t\in I$,
\begin{equation}\label{eq:bounded-pole-volume-comparison}
 0<c_{D,I,N}\le \mathfrak m_i(B_{\sqrt t}(z_i))
 \le C_{D,I,N}<\infty.
\end{equation}
The two-sided Gaussian bounds of \cite{JiangLiZhang} are used in both
directions: the upper bound controls $p_t^i$, while the lower bound controls
the negative part of $\log p_t^i$.  Together they imply
\begin{equation}\label{eq:entropy-tail-dominating-estimate}
 p_t^i(z_i,y)\,|\log p_t^i(z_i,y)|
 \le C_{D,I,N}e^{-c_{I,N}\mathsf d_i(z_i,y)^2}
       \bigl(1+\mathsf d_i(z_i,y)^2\bigr).
\end{equation}
Indeed, the two bounds control $|\log p_t^i|$ by
$C_{D,I,N}(1+\mathsf d_i(z_i,y)^2)$; the volume term in the bounds is absorbed by
\eqref{eq:bounded-pole-volume-comparison}.  Decomposing into unit annuli and
using standard volume comparison, the tail of the right-hand side of
\eqref{eq:entropy-tail-dominating-estimate} is bounded by
\[
 C_{D,I,N}\sum_{k\ge A-1}
 e^{-c_{I,N}k^2}(1+k^2)(k+2)^N,
\]
which tends to zero independently of $i$.  Thus
\begin{equation*}
 \lim_{A\to\infty}\sup_i\sup_{t\in I}
 \int_{X_i\setminus B_A(z_i)}
 p_t^i(z_i,y)\,|\log p_t^i(z_i,y)|\,\d\mathfrak m_i(y)=0
\end{equation*}
for every compact $I\Subset(0,\infty)$.  Truncation to bounded balls,
\eqref{eq:RCD-kernel-pointwise-convergence}, and measured convergence now
give convergence of $\mathcal N_{z_i}^{(N),i}(t)$.  The same sequential argument
with $t_i\to t\in I$ proves local uniformity in time.

For the second heat average the same estimates and standard volume comparison give
\begin{equation}\label{eq:pointed-entropy-quadratic-growth}
 \bigl|\mathcal N_y^{(N),i}(t)\bigr|
 \le C_{I,N}\bigl(1+\mathsf d_i(x_i,y)^2\bigr),
 \qquad t\in I.
\end{equation}
If $(t,s)$ ranges in a compact subset of $(0,\infty)^2$, both heat times are
uniformly positive.  Applying
\eqref{eq:RCD-kernel-pointwise-convergence} on bounded balls and using the
outer Gaussian bound to integrate the quadratic tail in
\eqref{eq:pointed-entropy-quadratic-growth} proves the second convergence in
\eqref{eq:RCD-entropy-heat-local}, locally uniformly in $(t,s)$.
\end{proof}

\begin{theorem}[Stability of inverse-time comparisons]\label{thm:stability}
Assume entropy--heat convergence.
\begin{enumerate}[label=\textup{(\roman*)}]
\item Exact forward codimension-$q$ comparison is closed under
entropy--heat convergence.
\item Fix $\sigma\in\mathbb R$ and $q\in[0,N]$.  If all $X_i$ satisfy the
scalar comparison with one common witness scale $\tau>0$, and, when $q>0$,
the codimension-$q$ comparison with one common witness scale $S>0$, then the
limit satisfies $\mathrm{HSCod}(\sigma,N,q)$ with the same witness scales.
\end{enumerate}
\end{theorem}

\begin{proof}
On a fixed heat-time leaf, entropy--heat convergence gives local uniform
convergence of the inverse-time Nash profiles.  Concavity on a fixed interval
is closed under local uniform convergence.  This proves (i), and the same
argument on the fixed regions $(0,\tau)^2$ and $\mathscr T_S^+$ proves (ii).
\end{proof}

\paragraph{Smooth limits.}
Let $N\ge3$ and let $(M_i^N,g_i,x_i)$ be complete, connected, noncompact
pointed manifolds with $\Ric_{g_i}\ge0$ and
$\Scal_{g_i}\ge\sigma_0>0$.  For fixed $\sigma<\sigma_0$ and $0<q<2$,
every entropy--heat limit of
$(M_i,\mathsf d_{g_i},c_i\Vol_{g_i},x_i)$, $c_i>0$, satisfies
$\mathrm{HSCod}(\sigma,N,q)$.  Indeed, the two parts of
\cref{thm:positive-scalar} give uniform witness scales, and
part~(ii) of \cref{thm:stability} applies.

\subsection{Euclidean space and macroscopic dimension}

\begin{proposition}[Euclidean space and tangent test]\label{prop:euclidean}
Let $0\le q\le N$.
\begin{enumerate}[label=\textup{(\roman*)}]
\item Equip $\mathbb R^m$ with $c\mathcal L^m$, $c>0$, and normalize the
entropy by $N\ge m$.  Then
\[
 \mathcal N_x^{(N)}(t)
 =\log c-\frac{N-m}{2}\log(4\pi t)-\frac{N-m}{2},
\]
and
\begin{equation}\label{eq:H-Euclidean}
 \mathcal H_{T,0}''(r)=-\frac{N-m}{2r^2}.
\end{equation}
Hence the exact forward codimension-$q$ comparison holds on $\mathbb R^m$
if and only if
\begin{equation}\label{eq:Euclidean-codimension}
 q\le N-m.
\end{equation}
\item Let $X$, with normalization dimension $N$, satisfy the exact forward
codimension-$q$ comparison.  If a pointed rescaling of $X$ converges in
entropy--heat sense to
$(\mathbb R^m,\mathsf d_{\mathrm E},c\mathcal L^m,0)$, then
\[
 m\le N-q.
\]
\end{enumerate}
\end{proposition}

\begin{proof}
For (i), substitute
$p_t(x,y)=c^{-1}(4\pi t)^{-m/2}e^{-|x-y|^2/(4t)}$ into
\eqref{eq:Nash-entropy}.  The entropy is spatially constant, and replacing
$t$ by $r^{-1}$ gives \eqref{eq:H-Euclidean} and
\eqref{eq:Euclidean-codimension}.  Thus $q$ is exactly the number of Euclidean
heat directions missing relative to the normalization dimension $N$.

For (ii), the exact forward comparison is invariant under scaling by
Proposition~\ref{prop:basic-properties} and passes to the limit by \cref{thm:stability}.
Part (i) then gives $q\le N-m$.
\end{proof}

\subsection{Blow-downs}

\begin{proof}[Proof of Theorem~\ref{thm:blowdown-dimension}]
Let $S_*$ be a macroscopic witness scale for $X$.  Fix $S>0$.  On the
rescaled space $(X,a_i\mathsf d,c_i\mathfrak m)$, the sector
$\mathscr T_S^+$ corresponds to $\mathscr T_{a_i^{-2}S}^+$ for the original
heat times.  Since $a_i^{-2}S\to\infty$, the coefficient-$q$ comparison holds
on $\mathscr T_S^+$ for all sufficiently large $i$.  Passing to the limit by
\cref{thm:stability} and then letting $S>0$ be arbitrary gives the exact
forward comparison on $X_\infty$.  The last assertion follows from
part~(ii) of Proposition~\ref{prop:euclidean}.
\end{proof}

\paragraph{A singular model.}
The macroscopic clause is nonempty on genuinely singular spaces.  Let
$(E^4,g_{\mathrm{EH}})$ be the Eguchi--Hanson space, which is complete and
Ricci-flat and is asymptotic to the flat cone $\mathbb R^4/\{\pm1\}$
\cite{Appleton2023}.  For
\[
 M=E^4\times S_a^2
\]
one has $\Ric_M\ge0$ and $\Scal_M=2/a^2$.  A normalized blow-down of $M$
converges in the pointed measured Gromov--Hausdorff sense to
$\mathbb R^4/\{\pm1\}$, with the spherical factor collapsing to a point.
The normalized blow-downs belong to the $\mathrm{RCD}(0,6)$ compactness
class, so \cref{prop:RCD-entropy-heat} upgrades this convergence to
entropy--heat convergence.  Applying \cref{thm:blowdown-dimension} for every
$q<2$ and then letting $q\uparrow2$ shows that the singular cone, with
normalization dimension $6$, satisfies the exact forward codimension-$2$
comparison.  This
example concerns the macroscopic clause; no scalar lower bound on the cone is
asserted.

\begin{remark}[Sharp large-time coefficient]
The Euclidean tangent test gives one inequality: an exact forward coefficient
$q$ forces every Euclidean tangent dimension $m$ detected by the theory to
satisfy $q\le N-m$.  In the products of \cref{rem:heat-codimension} this
bound is sharp at the level of the supremal large-time coefficient.  It is
natural to ask for geometric hypotheses under which the converse holds for
general blow-downs.
\end{remark}

\begin{corollary}[Smooth blow-down bounds]\label{cor:smooth-blowdown}
Let $(M^n,g)$ be complete, connected, and noncompact with $\Ric\ge0$.
Let $a_i\downarrow0$, $x_i\in M$, and
$c_i:=\Vol_g(B_{a_i^{-1}}(x_i))^{-1}$.  For any pointed measured
Gromov--Hausdorff limit $X_\infty$ of
$(M,a_i\mathsf d_g,c_i\Vol_g,x_i)$:
\begin{enumerate}[label=\textup{(\roman*)}]
\item if $\Rk_k\ge\sigma_0>0$ for some $2\le k<n$, then
      $\dim_{\mathrm{ess}}X_\infty\le k-1$;
\item if $\Scal\ge\sigma_0>0$, then
      $\dim_{\mathrm{ess}}X_\infty\le n-2$;
\item if $\mathcal C_m\ge\kappa>0$ for some $1\le m\le n-1$, then
      $\dim_{\mathrm{ess}}X_\infty\le m-1$.
\end{enumerate}
\end{corollary}

\begin{proof}
The normalized blow-down sequence lies in the finite-dimensional
$\mathrm{RCD}(0,n)$ compactness class.  Stability of the RCD condition gives
an $\mathrm{RCD}(0,n)$ limit; Proposition~\ref{prop:RCD-entropy-heat} gives
entropy--heat convergence, and Euclidean regular tangents are available by
\cite{MondinoNaber,BrueSemola}.  Apply
\cref{thm:blowdown-dimension} to \cref{thm:partial-ricci} for (i).  For (ii),
use part~(ii) of \cref{thm:positive-scalar} for every $q<2$ and let $q\uparrow2$.
For (iii), use \cref{thm:intermediate-curvature} for every $q<n-m+1$ and let
$q\uparrow n-m+1$.
\end{proof}

\end{document}